%% file: eres-arxiv.tex
\documentclass[]{amsart}
\pdfoutput=1
\usepackage{amssymb,amsmath}
\usepackage{xypic}
\usepackage{thmtools, thm-restate}
\usepackage{graphicx}
\usepackage{color}

\include{latex-header-arxiv}

\title{Amalgamation Classes with $\exists$-Resolutions}
\author{Justin Brody}

\address{Department of Mathematics and Computer Science\\
Goucher College\\
1021 Dulaney Valley Road\\
Baltimore, MD 21204}
\email{justin.brody@goucher.edu}
\begin{document}

\begin{abstract}
  Let $\K_d$ denote the class of all finite graphs and, for graphs $A
  \subseteq B$, say
  $A \leq_d B$ if distances in $A$ are preserved
  in $B$; i.e. for $a, a' \in A$ the length of the shortest path in
  $A$ from $a$ to $a'$ is the same as the length of the shortest path
  in $B$ from $a$ to $a'$.  In this situation $(\K_d, \leq_d)$ forms
  an {\it amalgamation class} and one can perform a Hrushovski
  construction to obtain a generic of the class.  One particular
  feature of the class $(\K_d, \leq_d)$ is that a closed superset of a
  finite set need not include {\it all} minimal pairs obtained
  iteratively over that set but only enough such pairs to resolve
  distances; we will say that such classes have {\it $\exists$-resolutions}.

  In \cite{moss}, Larry Moss conjectured the existence of graph $M$
  which was $(\K_d, \leq_d)$-injective (for $A \leq_d B$ any isometric
  embedding of $A$ into $M$ extends to an isometric embedding of $B$
  into $M$) but without finite closures.  We examine Moss's conjecture
  in the more general context of amalgamation classes.  In particular,
  we will show that the question is in some sense more interesting  in
  classes with $\exists$-resolutions and will give some conditions under
  which the possibility of such structures is limited.
\end{abstract}
\maketitle

\let\a\undefined
\newcommand{\a}{{
  \mathchoice{\forall}{\forall}{\scriptscriptstyle\forall}{\forall}
}}

\newcommand{\e}{{
  \mathchoice{\exists}{\exists}{\scriptscriptstyle\exists}{\exists}
}}

\section{Introduction}
This paper arises from an investigation of the class of {\it distanced
  graphs} first explored by Larry Moss in \cite{moss}.  In that paper,
Moss produced a {\it
  universal distanced graph}; this is a countable graph $U$ into which
every finite graph $A$ embeds {\it isometrically}.  That is, for any
finite graph $A$ there is
a map $f: A \into U$ such that for every $a, a' \in A$, the length of
the shortest path in $A$ from $a$ to $a'$ is the same as the length of
the shortest path in $U$ from $a$ to $a'$.

Moss's construction proceeded by expanding the language of graphs with
predicates $d_n(x,y)$ which were interpreted to indicate that the
(path-length) distance from $x$ to $y$ was precisely $n$.  While Moss
did not directly use a \fraisse\ limit to obtain his structure, a
number of his results can be recovered by doing so (see \cite{kuek}).

It is also possible to reinterpret Moss's work in the framework of the
Hrushovski construction.  Here we work in the language of graphs, let
$\K_d$ be the class of all finite graphs and specify that (for
$A, B \in \K_d$) $A \leq_d B$ just in case the inclusion map
$A \into B$ is an isometry.  This is an amalgamation class (see
Section \ref{dg-sect} or \cite{moss} for details) and can thus be
associated with a $(\K_d, \leq_d)$-generic.  The latter is the unique
(up to isomorphism) countable graph $G_d$ satisfying the following
conditions
\begin{itemize}
\item[] {\bf (Age):} Every finite subgraph of $G_d$ is in $\K_d$
\item[] {\bf (Injectivity):} If $f: A \into G_d$ so that $f(A)$ is
  isometric in $G_d$, 
  then for any $B$ with $A \leq_d B$, $f$ extends to an embedding $g:
  B \into G_d$ such that $g$ is an isometry.
\item[] {\bf (Finite Closures):}  For any finite subset $X \subseteq G_d$, there
  is a finite $Y$ with $X \subseteq Y \subseteq G_d$ and $Y$ is
  isometric in $G_d$
\end{itemize}
It is easy to show that in this situation $G_d$ is also a countable
universal distanced graph\footnote{It is in fact elementarily equivalent to
Moss's $U$ (see Theorem 4 of \cite{kuek}) and will consist of
countably many copies of $U$.}.  Moss
conjectured that there was a graph $M$ which satisfied the first two of
the three properties above but not the third.  For any amalgamation
class $(\K, \leq)$ we will call $M$ a {\it Moss structure} if it
satisfies appropriate analogues of the first two conditions but not
the third (see Definition \ref{defn-moss-struct}).  

There are three main results of this paper.  The first is that the
existence of a $(\K, \leq)$ Moss structure depends on a geometric
property of $(\K, \leq)$.  In particular, we will define what it means
for $(\K, \leq)$ to have $\forall$-closures (as most studied classes
do) versus $\exists$-resolutions (as the class of distanced graphs
$(\K_d, \leq_d)$ does).  We will show the following

\begin{restatable}{theorem}{mainthm}
\label{thm:main}\ 
  \begin{enumerate}
  \item If $(\K_\forall, \leq_\forall)$ is an amalgamation class with
    $\forall$-closures, then a $(\K_\forall, \leq_\forall)$ Moss
    structure exists exactly when the $(\K_\forall,
    \leq_\forall)$-generic is not $\omega$-saturated.   Further, if a
    $(\K_\forall, \leq_\forall)$ Moss structure exists, it can be
    taken to be a model of the theory of the $(\K_\forall, \leq_\forall)$-generic.
  \item If $(\K_\exists, \leq_\exists)$ is an amalgamation class with
    finitary $\exists$-closures, then no $(\K_\exists, \leq_\exists)$
    Moss structure exists which is a model of the the theory of the
    $(\K_\exists, \leq_\exists)$-generic.
  \end{enumerate}
\end{restatable}

We do not resolve the general question of the existence of Moss
structures for classes with $\exists$-closures, nor do we answer the
specific question of whether or not a $(\K_d, \leq_d)$ Moss structure
exists (although the second part of the above theorem implies that if
it does it cannot be a model of the theory of the generic).

We will also show that classes with $\e$-resolutions and full
amalgamation have superstable theories and prove a transfer theorem
for classes with $\e$-resolutions and related classes (which we will
call $\a$-companions).

\section{Amalgamation}

We will work throughout with amalgamation classes $(\K, \leq)$ in a
finite relational language $L$.  These are classes $\K$ of finite
$L$-structures partially ordered by the relation $A \leq B$ (read $A$
is {\it strong} or {\it closed} in $B$) with which one can produce a
canonical structure (the {\it generic} of the class) via an imitation
of the \fraisse\ construction.  We will specifically require our
classes to be closed under substructure and isomorphism, that $\leq$
be isomorphism invariant (if $A \leq B$ and $f: B \isom B'$, then
$f(A) \leq B'$) and to satisfy the following axioms of Baldwin and Shi
(\cite{bal-shi}):
\begin{itemize}
\item[\bf A1]\label{ax1} For $M \in \K$, $M \leq M$
\item[\bf A2]\label{ax2} For $M,N \in \K$, if $M \leq N$, $M \subseteq N$
\item[\bf A3]\label{ax3} For $A,B,C \in \K$, if $A \leq B \leq C$, then $A \leq C$
\item[\bf A4]\label{ax4} For $A,B,C \in \K$, if $A \leq C$ and $A \subseteq B \subseteq C$,
  then $A \leq B$
\item[\bf A5]\label{ax5} We have $\emptyset \in \K$ and
  $\emptyset \leq A$ for every $A \in \K$.
\end{itemize}

In most examples of the Hrushovski construction in the current
literature, the following additional axiom is also satisfied.
  \begin{itemize}
  \item[\bf A6]\label{ax6}  If $A \leq B$ and $X$ is embedded in a
    common superstructure of $B$, $A \cap X \leq B \cap X$
  \end{itemize}

  This axiom guarantees that in the $(\K, \leq)$-generic the smallest
  closed superset of a finite set is uniquely defined.  Our primary
  interest in this paper is in examining amalgamation classes which do
  {\it not} satisfy {\bf A6}.

  In order to produce a generic of $(\K, \leq)$, we also require that
  the class have the {\it amalgamation property}, defined below.  We
  quickly summarize the main ideas of the construction; see (e.g.)
  \cite{bal-shi, mcl-kuek, wagner} for more details.

\begin{notn}\ 
  \begin{itemize}
  \item We will write $A \finsubs B$ to indicate that $A$ is a {\bf
      finite} substructure of $B$.
  \item We will write $XY$ to denote $X \cup Y$.
  \end{itemize}

\end{notn}
  % arxiv only
  \def\sinto{\overset{{\leq}}{\into}}
\begin{defn}

  Let $(\K, \leq)$ be a class of $L$-structures partially ordered by
  $\leq$ satisfying A1--A5 as above.
  \begin{enumerate}
  \item For any $L$-structure $M$, the set of $M$'s finite substructure
    is referred to as the {\it age} of $M$.
  \item If $M$ is an $L$-structure whose age in contained in $\K$, we
    say that $\K$ is {\it cofinal} in $M$.
  \item If  $\K$ is cofinal in $M$ and $A \finsubs M$, we say that $A
    \leq M$ if for every $X$ with  $A \subseteq X \finsubs M$, $A \leq X$.
  \item If $M$ has $\K$ cofinal in it and  $f:  A \into M$ is an
    embedding, we say that $f$ is a {\it strong} or {\it closed}
    embedding if $f(A) \leq M$.  We write this as $f:  A \sinto B$
  \item We say that $(\K, \leq)$ has the {\it amalgamation property}
    if for $f:  A \sinto B, g: A \sinto C$, there is a structure $D
    \in \K$ and strong embeddings $f':  B \sinto D, g':  C \sinto D$ such that
    $f' \circ f = g' \circ g$.  In other words, given the solid part
    of the following diagram it  can be completed to commute

\centerline{ \xymatrix{
    &B \ar@{^{(}.>}_{f'}^{\leq}[rd]    \\
    A \ar@{^{(}->}_{f}^{\leq}[ru] \ar@{^{(}->}_{\leq}^{g}[rd] &&D \\
    &C \ar@{^{(}.>}_{\leq}^{g'}[ru] }} We will call $D$ an {\it
  amalgam} of $B$ and $C$ over $A$.
  \end{enumerate}
\end{defn}

We will call $(\K, \leq)$ an {\it amalgamation class} if it satisfies
A1--A5 and has the amalgamation property.

In fact, we will often want to work with especially nice forms of
amalgamation.  
\begin{defn} Suppose $A,B,C$ are elements of $\K$ with $A = B \cap C$,
  and let $D$ be the structure whose universe is $BC$ and whose
  relations are precisely those of $B$ and those of $C$.  Then we will
  denote $D$ by $B \oplus_{A} C$
  
  \begin{itemize}
  \item If $(\K, \leq)$ is an amalgamation class in which $B
    \oplus_{A} C$ is an amalgam of $B$ and $C$ over $A$, then we will
    call $B \oplus_{A} C$ the {\it free amalgam} of $B$ and $C$ over
    $A$ and say that $(\K, \leq)$ is a {\it free amalgamation class}.

  \item A free amalgamation class is {\it full} if for $A,B,C\in\K$
    and $A\le B$, $ A\subseteq C$, then $C\le D$, where
    $D=B \oplus_A C$.
  \end{itemize}
\end{defn}

Given an amalgamation class $(\K, \leq)$, one can imitate the
\fraisse\ construction to produce a {\it generic} of the class.  This
is the unique (up to isomorphism) countable $L$-structure $G$ which
satisfies:

\begin{enumerate}
\item[] {\bf (Age):}   $\K$ is cofinal in $G$.
\item[] {\bf (Injectivity):} For $A, B \in \K$, if $f: A \sinto G$ and
  $A \leq B$, then $f$ extends to some ${\hat f}: B \sinto G$.
\item[] {\bf (Finite closures):} For every $A \finsubs G$, there is a
  finite $B$ with $A \subseteq B \leq G$.
\end{enumerate}

We will say that any structure which satisfies the first two of these
conditions is $(\K, \leq)$-injective.

It turns out that the model theory of the generic is largely
determined by the complexity of the closure operation ( that is,  by the
complexity of finding a minimal finite superset $B$ of a given $A \finsubs G$
for which $B \leq G$).  We introduce some notions for analyzing such
supersets.  The fundamental one is that  of a minimal pair, which is a
minimal example of an extension which is not strong.

\begin{defn} Let $(\K, \leq)$ be an amalgamation class.
  \begin{enumerate}
  \item   For $X, Y \in \K$ with $X \subseteq Y$, we say that $(X,Y)$
    is a {\it minimal pair} if $X \not  \leq Y$ but $X \leq  Y_0$ for
    $X \subseteq Y_0 \subsetneq Y$.  We will also call $Y$ a {\it
      minimal extension} of  $X$.
  \item If $(X,Y)$ is a minimal pair, we say that it is a {\it
      biminimal pair} if whenever $X_0 \subseteq X, Y_0 \subseteq Y
    \setminus X$ and $(X_0, X_0 Y_0)$ is a minimal pair, we must  have
    $X = X_0$ and $Y = Y_0$.  We will also say that $Y$ is  a {\it
      biminimal extension} of $X$.
  \end{enumerate}
\end{defn}

In \cite{bal-shi}, Baldwin and Shi  noted that in classes satisfying
A6, for $A, B \in \K$ we have $A \leq B$ exactly when for $X \subseteq
A, Y \subseteq B$ and $(X,Y)$ a biminimal pair, we must have $Y
\subseteq B$.  This will not hold as biconditional without A6, but
does give a sufficient condition.

\begin{lem}\label{no-bmp-suff}
  If $(\K, \leq)$ is an amalgamation class (and in particular satisfies
  A1-A5) and $A \subseteq B$ are structures from $\K$ such that for
  every biminimal pair $(X,Y)$ with $X \subseteq A, Y \subseteq B$ we
  have $Y \subseteq A$, then $A \leq B$.
\end{lem}

\begin{proof}
  Suppose $A \not \leq B$.  If $(A,B)$ is not a biminimal pair then by
  definition there must be some $X \subseteq A$ and $Y \subseteq B$
  such that $(X,Y)$ is a biminimal pair.
\end{proof}

\section{Distanced Graphs}\label{dg-sect}
In this section we will formally explore that class of {\it distanced
  graphs} mentioned in the introduction.  This class was studied by
Moss in \cite{moss} and will form our canonical example of a class
that does not satisfy A6 but does have what we will
$\exists$-resolutions.  The remainder of the paper will examine this
more general property, with a special eye toward the question of the
existence of Moss structures.

\begin{defn}
  Let $G$ be a graph.  For $a, b \in G$, a {\it path} from $a$ to $b$
  in $G$ is a set of vertices $a = p_0, \ldots p_n = b$ where:
  \begin{enumerate}
  \item $p_i \in G$
  \item there is an  edge from $p_i$ to $p_{i+1}$ for $i < n$
  \item for $i, j < n$, $p_i = p_j$ only if $i=j$ (thus we consider
    all paths to be simple)
  \end{enumerate}
The {\it length} of a path is the number of edges in the path.
\end{defn}

As before, let $\K_d$ be the class of all finite graphs and say
$A \leq_d B$ when for every $a,a' \in A$, the minimal path length from
$a$ to $a'$ in $A$ is the same as the minimal path length from $a$ to
$a'$ in $B$.  It is easy then to verify that A1 -- A5
hold, and the ``strong amalgamation lemma'' of \cite{moss} shows
that $(\K_d, \leq_d)$ has free amalgamation.  We show that in fact it
has full amalgamation.

To that end, fix $A, B, C \in \K_d$ with $A \leq_d B, A \subseteq C$
and let $D = B \oplus_A C$.  We have to show that $C \leq_d D$
($D \in \K_d$ by definition).  That is, for $x,y \in C$ we have to
show that any minimal length path from $x$ to $y$ in $C$ has the same
length as a minimal length such path in $D$.  Let $p$ be a minimal
length path from $x$ to $y$ in $D$ and suppose by way of contradiction
that it is shorter than any path from $x$ to $y$ in $C$.  List the
vertices of $p$ as $x = x_0, \ldots, x_n = y$ where there is an edge
between each $x_i$ and $x_{i+1}$.  Since the amalgam is free, we can
choose pairs of indices $(s_j, e_j)$ such that:
\begin{itemize}
\item $s_j < e_j$ for all $j$
\item $p_{s_j}, p_{e_j} \in A$
\item For $s_j < k < e_j$, we have $p_k \in B \setminus A$
\end{itemize}
Since $p$ shorter than any path from $x$ to $y$ in $C$, we must have
that for some $j$, $p_{s_j}, p_{s_j + 1}, \ldots p_{e_j}$ is shorter
than any path from $p_{s_j}$ to $p_{e_j}$ in $A$.  This contradicts
that $A \leq_d B$.

Let $M$ be the $(\K_d, \leq_d)$-generic.  Then we can define a closed
superset of $A \finsubs M$ by recursively adding minimal length paths
over subsets of $A$.  In particular, for $x, y \in M$ we
define $\chi(x,y)$ to be any minimal length path from $x$ to $y$.
Then let $J_0(A) = \bigcup \{\, \chi(x,y) : x,y \in A  \,\}$ and having
defined $J_n$ let $J_{n+1}(A) = \{\, \chi(x,y) : x,y \in J_n(A)  \,\}$.
Letting $B = \bigcup_{n \in \omega} J_n$ it is easy to see that $A
\subseteq B \leq_d M$.  Thus we can define a minimal closed superset of
$A$ by inductively adding to $A$ a single minimal pair $(xy,
\chi(xy))$ for each $x,y \in A$.  This idea will form the basis for
our definition of a class with $\e$-resolutions. 

\begin{lem}
  For any finite graph $X$, the pair $(X,Y)$ is a
  biminimal pair in $(\K_d, \leq_d)$ if and only if $X$ is a pair of
  points not joined by an edge and $Y$ is a path between them.
\end{lem}

\begin{proof}
  Clear.
\end{proof}

\section{$\exists$-Resolutions}
In this section we formally define the notion of having
$\exists$-resolutions and examine its consequences.  The notion will
correspond to being able to form a closed superset of a finite base by
iterative adding {\it some} minimal pair extensions which occur over
the base.  This is in contrast to Baldwin and Shi's building of the
intrinsic closure by iteratively adding {\it all} minimal pair
extensions which occur over a given a base.

\begin{defn}
  Let $(\K, \leq)$ be an amalgamation class satisfying the axioms
  A1 through A5 above.  Following Baldwin and Shi \cite{bal-shi}, 
  for any $M$ cofinal with $\K$ and $A \finsubs M$, we
  define the {\bf maximal} closure ($\mcl$) as follows
  \begin{itemize}
  \item $I_0(A) := A \cup \bigcup \{\, B \subseteq M : \exists A_0 \finsubs A
   \text{ with } (A_0, B) \text{ a minimal pair}\ \,\}$
  \item $I_{n+1}(A) := \bigcup \{\, B \subseteq M : \exists A_0 \finsubs I_n(A)
   \text{ with } (A_0, B) \text{ a minimal pair}\ \,\}$
 \item $\mcl_M(A) := \bigcup_{n \in \omega} I_n(A)$
  \end{itemize}
\end{defn}

It is worth noting that under axiom A6 $\mcl_M(A)$ is precisely the
closure of $A$ in $M$.  In general it is clear that $\mcl_M(A) \leq M$
(by Lemma \ref{no-bmp-suff}).  If $\mcl_M(A)$ is always the smallest
closed superset of $A$ which is closed in $M$, then we say that
$(\K, \leq)$ has {\it ${\forall}$-closures} or {\it $\a$-resolutions}.
Intuitively, in these classes minimal pairs can be thought of as
obstructions to a set's being closed, and all such obstructions must be
included in the closure of a finite set.
\begin{lem}  
  For any amalgamation class $(\K, \leq)$ satisfying  A1 -- 
    A5, Axiom A6 holds if and only if $(\K, \leq)$ has
  $\forall$-closures.
\end{lem}

\begin{proof}
  Suppose $(\K, \leq)$ has $\forall$-closures.  Let $A \leq B$ and let
  $C$ be any element of $\K$ which is in a common superstructure of
  $A,B$ (so that the intersections $A \cap C, B \cap C$ make sense).
  We have to show that $A \cap C \leq B \cap C$.  Let $(X,Y)$ be a
  minimal pair with $X \subseteq A \cap C$ and $Y \subseteq B \cap C$.
  Then $(X,Y)$ is also a minimal pair with $X \subseteq A$, so by
  $\forall$-closures $Y \subseteq A$.  Since $Y$ was assumed to be
  contained in $Y \cap C$, $Y$ is contained in $A \cap C$ as needed.

  Conversely, if $(\K, \leq)$ does not have $\forall$-closures, then
  there is a finite $A$ and some $M$ in which $\K$ is cofinal such that $A \leq
  M$ but for some minimal pair $(X,Y)$ with  $X \subseteq A$ we have
  $Y \not \subseteq A$.  Let $C = \mcl_M(A)$; then $A \leq C \leq
  M$.  However $A \cap Y \not \leq C \cap Y$ since $A \cap Y \supseteq
  X$ and $C \cap Y = Y$.

\end{proof}

For any $M$ as above, let us call $B$ a {\bf resolution} of $A$ in $M$
if $A \subseteq B \leq M$.  Such a resolution will be called {\bf
  minimal} if there no resolution $B'$ of $A$ with $B' \subseteq B$.  Minimal
resolutions will also be called {\bf closures}; note that in the
absence of axiom A6, minimal closures need not be unique while in
classes with $\forall$-closures the notion is precisely that of the
$M$-closure of $A$.

\begin{defn}
  Let us say that a reflexive binary relation $R$ on a set $Z$ {\it induces
  a partial order} if, letting
  $[a]_R := \{\, b \in Z : aRb \text{ and } bRa \,\}$, and defining
  $[a]_R R' [b]_R$ whenever $a_0 R b_0$ for some
  $a_0 \in [a]_R, b_0 \in [b]_R$, we have that $R'$ is well defined
  and is a partial order on the equivalence classes of $Z$.
\end{defn}

\def\tl{\unlhd}
Note that $R$ induces a partial order on $Z$ exactly when there is a
map $f:  Z \to (P, \tl)$ where $(P, \tl)$ is partial order and for
$z,z' \in Z$, $z R z'$ exactly when $f(z) \tl f(z')$.

\begin{defn}\label{ex-cl-defn}
  We say that a class $(\K, \leq)$ satisfying A1-A5 has
  {\it $\exists$-resolutions} if for any $X$, there is a relation $\epo_X$
  which induces a partial order on the
  class $\{\, Y : (X, Y) \text{ is a biminimal pair} \,\}$ such that
  for $A \subseteq B \in \K$, we have $A \leq B$ if and only if for
  every $X \subseteq A$, if $(X,Y')$ is a biminimal pair with $Y'
  \subseteq B$ then there is a biminimal pair $(X, Y)$ with $Y
  \subseteq A$ and $Y \epo_X Y'$.

  A class $(\K, \leq)$ with $\exists$-closures will further be said to
  be {\it coherent} if for any $(\K, \leq)$-biminimal pair $(X,Y)$ if
  \begin{enumerate}
  \item There is some $Z$ with $X \subseteq Z \subseteq Y$ such
    that there is a biminimal pair $(U,V)$ with $U \subseteq Z$ and
    $V = Y \setminus Z$; and
  \item  there is some $V'$ with $(U, V')$ a biminimal pair; and
  \item  $V' \epo_U V$
  \end{enumerate}
  Then $(X, ZV')$ is a biminimal pair and $ZV' \epo_X ZV$; i.e. $ZV'
  \epo_X Y$.

%   \setminus V$ 
% V = 

%   and
%   $Y_0$ with $X \subseteq Y_0 \subseteq Y$ we have some $Y'$ such that
%   \begin{itemize}
%   \item $(Y_0, Y')$ is a biminimal pair
%   \item $(Y_0, Y)$ is a biminimal pair
%   \item $Y' \epo_{Y_0} Y$
%   \end{itemize}
%   then we have that $(X, Y')$ is a biminimal pair and $Y' \epo_{X} Y$
%   as well.
\end{defn}

Intuitively, the coherence of a class indicates that if we swap a part
of a biminimal extension with something smaller with respect to a
subextension, then the resulting biminimal extension will be smaller
as well.  In many of our examples the partial order will be induced by
a number-valued function and coherence will come from a kind of
monotonicity of that function.\footnote{In fact, the notion of
  coherence can be seen as a generalization of the prinicple of
  optimalilty in computer science.  The latter states that optimal
  solutions to a global problem are comprised of optimal solutions to
  local problems.  Here global optimality would correspond to
  $\epo_X$-minimality while local optimality would correspond to
  $\epo_{U}$-minimality}

As for classes with $\a$-closures, we want to extend the defintion of
$\leq$ to arbitrary structures $M,N$ in which $\K$ is cofinal.
\begin{defn}
  If $(\K, \leq)$ is an amalgamation class with $\e$-resolutions and
  $\K$ is cofinal in $M,N$ with $M \subseteq N$, then we say that
  $M \leq N$ when for any biminimal pair $(X,Y)$ with
  $X \finsubs M, Y \finsubs N$, there is some $Y' \epo_X Y$ with
  $Y' \subseteq M$.
\end{defn}

%This is the old definition; not sure it really buys us anything.  

In analogy with Baldwin and Shi's characterization of the intrinsic
closure in classes which satisfy A6, we offer the following

\begin{lem}\label{lem-compiled-ec}
  If $(\K, \leq)$ is a free amalgamation class which satisfies A1-A5,
  then it has $\exists$-closures if and only if for any $M$ in which
  $\K$ is cofinal, there are functions
  $\pi_M: [M]^{< \omega} \rightarrow 2^{[M]^{< \omega}}$ (a
  ``potential function'') and
  $\chi_M: [M]^{< \omega} \to [M]^{< \omega}$ (a ``choice function'')
  $M$ such that 
  \begin{enumerate}
  \item For $X \finsubs M$, $\pi_M(X)$ is a subset of $\{\, Y
    \subseteq M : (X, Y) \text{ is a biminimal pair} \,\}$  
  \item For $A \finsubs M$, $\chi_M(A) \in \pi(A)$  such that $A_\chi \leq M$, where $A_\chi$ is defined inductively
    as follows
    \begin{itemize}
    \item $J_0(A) := A \cup  \bigcup \{\, \chi_M(A_0) \text{ for } A_0
      \finsubs A\,\}$
    \item $J_{n+1}(A) := J_n(A) \cup  \bigcup \{\, \chi_M(A_0) \text{ for } A_0
      \finsubs J_n(A) \,\}$
    \item $A_\chi := \bigcup_{n \in \omega} J_n(A)$
    \end{itemize}
  \end{enumerate}
\end{lem}

\begin{proof}
  Suppose $(\K, \leq)$ has $\e$-resolutions.  Fix $M$ as above and
  for $X \finsubs M$, let $\pi(X)$ be the set of $\epo_X$-minimal
  extensions that occur over $X$ (or $\{\, \emptyset \,\}$ if there
  are no biminimal extensions), and choose $\chi(X)$ arbitrarily
  from $\pi(X)$.  Then it is clear that $A_\chi \leq M$ as desired.

  Conversely, suppose that our condition is satisfied for any $M$ with
  age $\K$.  Let $M$ be the $(\K, \leq)$-generic.  For $X \in \K$ and
  biminimal extensions $Y, Y'$, let $D = Y \oplus_X Y'$ and choose a
  strong embedding of $D$ into $M$; abusing notation we will identify
  $D$ with its image in $M$.  Let us say $Y \epo_X Y'$ exactly when
  $Y \leq M$.  Note that in such a case we will have $Y \leq D$ and
  that the isomorphism invariance of $\leq$ guarantees that this
  definition does not depend on the particular embedding of $D$ into
  $M$.

  We must show that this induces a partial order.  Reflexivity and
  antisymmetry come from the definition of the equivalence classes
  induced by $\epo_X$, noting that we must have
  $\chi_D(X) \in \{\, Y, Y'\,\}$.  For transitivity, we have to show
  that if $Y \epo_X Y'$ and $Y' \epo_X Y''$, then $Y \epo_X Y''$.  Let
  $D = ( Y \oplus_X Y') \oplus_X Y''$ and assume $D$ is a strong
  substructure of $M$.  We assume that each of $Y, Y', Y''$ are in
  different $\epo_X$-equivalence classes and we will show that
  $Y \leq M$.  First note that $\chi_D(X)$ must be exactly one of
  $Y, Y'$ or $Y''$; if not then letting $\chi_D(X) = Z$ we may assume
  that $X \leq (Z \cap Y), X \leq (Z \cap Y')$ and
  $X \leq (Z \cap Y'')$ so by free amalgamation $X \leq Z$, a
  contradiction.  Further, we must have that $X \chi_D(X) \leq D$,
  since by biminimality and free amalgamation we cannot have
  $\chi_D( \chi_D(X)) \in D$.  Thus $X \chi_D(X) \leq M$.  If
  $\chi_D(X) = Y'$, then we have $XY' \leq XYY'$ contradicting that
  $Y,Y'$ are in different $\epo_X$ classes.  Similarly if
  $\chi_D(X) = Y''$.  So we must have $\chi_D(X) = Y$ and
  $XY \leq XYY'$ as desired.
\end{proof}

Intuitively, in classes with $\exists$-resolutions minimal pairs can be
thought of as resolving some undetermined property of $A$.  Once that
property is resolved other potential resolutions add nothing new; thus
one only needs at most one biminimal extension of any finite subset in
a resolution of a finite set $A$.

Note that in such a situation the definition of $\chi$ may not be
preserved by substructures: that is, one might have a class with
$\exists$-closures and models $M \subseteq N$ with $A_\chi$ in $M$
defined differently from $A_\chi$ in $N$.  In fact, it is easy to see
that if $\chi_M(A) = \chi_N(A)$ for every $A \finsubs M$, then
$M \leq N$ (this implication does not reverse since a strong
substructure could allow arbitrary choices of $\chi_M$; the possible
choices $\pi_M$ should be preserved however).  Also note that in
general $\chi$ need not be unique.  For example, in the class of
distanced graphs discussed above, $\chi$ can be chosen so that for
$a,b \in A$, $\chi(\{\, a,b \,\})$ is {\it any} minimal length path
from $a$ to $b$ in $M$.

\subsection{Companions}
We now turn our attention to classes with $\e$-resolutions which are
derived from a given class.

\begin{defn}
  For a given class $(\K, \leq)$, a class $(\K_\exists, \leq_\exists)$
  is an $\exists$-companion of $(\K, \leq)$ if:
  \begin{enumerate}
  \item $\K_\exists = \K$
  \item $(X,Y)$ is a biminimal pair in $(\K, \leq)$ exactly if $(X, Y)$
    is a biminimal pair in $(\K_\exists, \leq_\exists)$
  \item $(\K_\exists, \leq_\exists)$ has $\exists$-closures.
  \end{enumerate}
\end{defn}

It is natural to wonder when $\exists$-companions exist.  For example,
the class of Shelah-Spencer graphs are discussed in \cite{bal-shi,
  bal-shel, mcl}.  For a fixed irrational $\alpha \in (0,1)$, the
class $(\K_\alpha, \leq_\alpha)$ is defined by saying that for graphs
$A \subseteq B$, $A \leq_\alpha B$ exactly when, letting $e(Y/X)$
denote the number of edges in $Y$ but not in $X$, we have
$|B_0 \setminus A| - \alpha e(B_0 / A) > 0$ for every
$A \subseteq B_0 \subseteq B$.  We then let
$\K_\alpha := \{\, A : \emptyset \leq_\alpha A \,\}$.  A biminimal
pair $(X,Y)$ will have $|Y \setminus X| - \alpha e(Y/X) < 0$; if we
had an $\e$-companion with free amalgamation we would be able to form,
for $n \in \omega$, $D_n$ as the free amalgam of $X$ with $n$ copies
of $Y$ and have $D_n \in \K_\alpha$ (one copy would form the
resolution of $X$ and the others would be then be strong extensions of
the pair).  But for sufficiently large values of $n$,
$|D_n| - \alpha e(D_n / \emptyset)$ would be negative, contradicting
the definition of $\K_\alpha$.  On the other hand, this is the main
obstruction as the following lemma shows.

\begin{lem}
  Let $(\K, \leq)$ be an amalgamation class satisfying A1--A6, and let
  $(\K', \leq')$ be formed from $(\K, \leq)$ by imposing an arbitrary
  partial ordering on the biminimal extensions of a fixed $X \in \K$
  and defining $\leq_\e$ in accordance with the definition of a class
  with $\exists$-closures.  Then if $(\K', \leq')$ is an amalgamation
  class it is an $\exists$-companion for $(\K, \leq)$.
\end{lem}
\begin{proof}
  One need only check that axioms A1--A5 hold, but this is straightforward.
\end{proof}

For example, it is straightforward to show that the class of distanced
graphs $(\K_d, \leq_d)$ is an $\exists$-companion to the class
$(\K_{C}, \leq_{C})$ of all graphs with $A \leq_{C} B$ exactly when
every path from $a$ to $a'$ with length at least $2$ in $B$ is
contained in $A$.  For a fixed pair $X = \{\, a, a' \,\} \in \K_C$
with no edge $(a, a')$, the biminimal pairs over $X$ will be the set
of all paths between $a$ and $a'$, with the path length inducing the
relevant partial ordering.  A consequence of the previous lemma is
that we could change the partial ordering to achieve a different
$\exists$-companion, provided that the resulting class is an
amalgamation class, as we will see with example \ref{hamilton-ex}.

\begin{lem} Let $(\K, \leq)$ be a full amalgamation class with
  $\forall$-closures.  If $(\K_\exists, \leq_\exists)$ is an
  $\exists$-companion for $(\K, \leq)$ with free and coherent
  amalgamation, then $(\K_\e, \leq_\e)$ is also a full amalgamation
  class.
\end{lem}

\begin{proof}
  We will show that for $A \leq_\e B$ and $A \subseteq C$ we have
  $C \leq_\e D = B \oplus_A C$.  Let $(X, Y)$ be any $(\K_\e, \leq_\e)$
  biminimal pair with $X \subseteq C$ and $Y \subseteq D$.  We have to
  show that there is some $Y'$ with $Y' \epo_X Y$ and
  $Y' \subseteq C$.  We can assume without loss that $X \cap C$ and $Y
  \cap C$
  are not both empty; our proof will be a generalization of the proof
  of full amalgamation for distanced graphs.

  Let $Z = (Y \cap C)$.  Since $(X,Y)$ is a minimal pair, we must have
  that $(Z, Y)$ is as well and in particular that there is a biminimal
  pair $(U, UV)$ with $U \subseteq Z, V \subseteq Y \setminus Z$.  In
  fact, I claim that $U \subseteq A$.  If not, then
  $U \cap (C \setminus A) \neq \emptyset$. Thus $U \cap A \leq UV$ (by
  biminimality) and $U \cap A \subseteq U$, so that by full
  amalgamation we have $U \leq (V \cap B) \oplus_{U \cap A} U = V$, a
  contradiction.  Thus we have $U \subseteq A$ as desired.

  Given this, we have a biminimal pair $(U, UV)$ with $U \subseteq A$
  and $V \subseteq (B \setminus A)$.  Since $(X,Y)$ is biminimal, we
  must have $V \setminus X = Y \setminus X$ (if
  $V \setminus X = Y_0 \subsetneq Y \setminus X$, then $(Z,Y)$ a
  minimal pair implies $Z \leq Y_0$ so
  $Z \cap (UV Y_0) \leq Y_0 \cap (U V Y_0)$ so $U \leq V$ a
  contradiction).  Thus $(U, V)$ is a biminimal pair with $U \subseteq
  A$.  

  Since $A \leq_\e B$, we have some $V' \epo_{U} V$ with
  $V' \subseteq A$.  By coherence we must have $ZV' \epo_X Y$.
\end{proof}

We can also pass from classes with $\e$-resolutions to classes with
$\a$-closures.

\begin{defn}
  Let $(\K, \leq)$ be any amalgamation class satifying A1 -- A5.  We say that an amalgamation class
  $(\K_\forall, \leq_\forall)$ is a $\forall$-companion of
  $(\K, \leq)$ if
  \begin{enumerate}
  \item $\K_\forall = \K$
  \item $(X,Y)$ is a biminimal pair in $(\K, \leq)$ exactly if $(X, Y)$
    is a biminimal pair in $(\K_\exists, \leq_\exists)$
  \item $(\K_\forall, \leq_\forall)$ has $\forall$-closures.
  \end{enumerate}
\end{defn}

The $\forall$-companion always exists and is easily characterized.

\begin{lem}
  For any class $(\K, \leq)$, a class $(\K', \leq')$ is an
  $\forall$-companion of $(\K, \leq)$ if and only if $\K' = \K$ and
  $A \leq' B$ exactly when for $(X,Y)$ a $(\K, \leq)$ minimal-pair
  with $X \subseteq A$ and $Y \subseteq B$ we have $Y \subseteq A$ as
  well.
\end{lem}

\begin{proof}
Suppose $(\K', \leq')$ is a $\forall$-companion of $(\K, \leq)$.  Then
it is clear that $\K' = \K$, we have to show that the condition on $A
\leq' B$ holds.  But this is clear since the notion of biminimal pairs
stays the same in an $\a$-companion and the condition on $\leq'$ is
then equivalent to having $\forall$-closures by definition.

For the converse, suppose that $(\K', \leq')$ is as described.  The
first and third conditions of being a $\forall$-companion are clear;
we need to show that the notions of biminimal pair coincide. We first
show that $(\K, \leq)$ biminimal pairs are $(\K', \leq')$ {\it
  minimal} pairs and vice versa.  We then show that they are in fact
biminimal pairs.

Suppose $(X, Y)$ is a $(\K, \leq)$ biminimal pair.  We want to show
that $X \not \leq' Y$ and for $X \subseteq Y_0 \subsetneq Y$
$X \leq' Y_0$.  That $X \not \leq' Y$ is clear from the definition of
$\leq'$.  If $X \not \leq' Y_0$, then there is a $(\K, \leq)$ minimal
pair $(U, V)$ with $U \subseteq X, V \subseteq Y_0$ and
$V \not \subseteq X$.  By biminimality, we must have $U=X, V=Y$,
contradicting that $Y_0 \subsetneq Y$.  Thus any $(\K, \leq)$
biminimal pair is a $(\K', \leq')$ minimal pair.

Suppose now that $(X, Y)$ is a $(\K', \leq')$ biminimal pair.  We have
to show that $X \not \leq Y$ and $X \leq Y_0$ for
$X \subseteq Y_0 \subsetneq Y$.  If $X \leq' Y$ then for every
$(\K, \leq)$ minimal pair $(U,V)$ with $U \subseteq X, V \subseteq Y$
we have $V \subseteq X$ as well.  But by lemma \ref{no-bimin-suff}
this implies that $X \leq Y$ as well. The same lemma implies that
$X \leq Y_0$ as well (since $X \leq' Y_0$).  Thus any $(\K', \leq')$
biminimal pair is also a $(\K, \leq)$ minimal pair.

Now suppose that $(X,Y)$ is a $(\K, \leq)$ biminimal pair.  By the
above, it is also a $(\K', \leq')$  minimal pair.  If there is a
$(\K', \leq')$ minimal pair $(U, V)$ with $U \subseteq X, V \subseteq
Y$, then there must be a $(\K', \leq')$ biminimal pair $(U_0, V_0)$
with $U_0 \subseteq U, V_0 \subseteq V$.  Since $(U_0, V_0)$ is also
$(\K, \leq)$ minimal, we must have $U_0 = X = U$ and $V_0 = Y = V$ and
$(X,Y)$ is $(\K', \leq')$ biminimal as
required.  The same argument, {\it mutatis mutandis}, shows that
$(\K', \leq')$ biminimal pairs are also $(\K, \leq)$ biminimal.
\end{proof}

\begin{cor}
  For any $(\K, \leq)$ as above, the $\forall$-companion
  $(\K_\forall, \leq_\forall)$ exists and is unique
\end{cor}

\begin{proof}
  The previous lemma shows us that if we define
  $(\K_\forall, \leq_\forall)$ to have $\K_\forall = \K$ and
  $A \leq_\a B$ whenever for $(X,Y)$ a $(\K, \leq)$ minimal pair with
  $X \subseteq A, Y \subseteq B$ we have $X \subseteq A$ as well, then
  $(\K_\forall, \leq_\a)$ will be an $\a$-companion as long as it is
  an amalgamation class.  It is easily checked that $(\K_\a, \leq_\a)$
  so defined will satisfy A1--A6.  Uniqueness also follows from the
  previous lemma.

\end{proof}

\def\L{\bold{L}} \def\sqleq{\sqsubseteq} For amalgamation classes
$(\K, \leq)$ and $(\L, \sqleq)$, we will say that $(\K, \leq)$ is
isomorphic to $(\L, \sqleq)$ if there is a bijection $F: \K \to \L$
such that for $A, B \in \K$, $F(A) \sqleq F(B)$ exactly when
$A \leq B$.  Note that it sufficient to have $\K = \L$ and for
$\leq, \sqleq$ to induce the same notion of biminimal pair.

\begin{lem}
\def\L{\bold{L}}
\def\sqleq{\sqsubseteq}
  For any class $(\K, \leq)$ with $\forall$-closures, the
  $\forall$-companion of any $\exists$-companion of $(\K, \leq)$ is
  isomorphic to $(\K, \leq)$.  
\end{lem}

\begin{proof}
  Let $(\K', \leq')$ denote the $\forall$-companion of some
  $\exists$-companion $(\K, \leq_\exists)$ of $(\K, \leq)$.
  It suffices to show that $(X,Y)$ is a biminimal pair in $(\K, \leq)$
  if and only if $(X,Y)$ is a biminimal pair in $(\K, \leq')$.  But this
  is clear from the definition.
\end{proof}

As the examples which follow will show, a single $\a$-companion can be
associated with multiple (non-isomorphic) $\e$-companions.

\subsection{Examples}

\begin{ex}\label{path-example}
  \def\edge{\sim}
  Let $L$ be the language whose non-logical symbols are $E(x,y)$ and
  $<$.  Let $\K_p$ be the class of all finite $L$-structures in which
  $E$ and $<$ are interpreted so that for $A \in \K_p$
  \begin{itemize}
  \item $E^A$ is symmetric and antireflexive (thus $(A, E)$ is a
    simple graph).
  \item $(A, <)$ is a finite linear order.
  \item For $x,y \in A$, if $A \models E(x,y)$ then either $y$ is the
    successor of $x$ in $(A, <)$ or vice-versa.
  \end{itemize}
  Thus we can view $A$ as a finite linear order with possible edges
  between pairs of successive vertices.  For $x \in A$, we will write
  $x \pm 1$ for the successor or predecessor of $x$ in $A$ (if there
  is one).

  For $A \subseteq B \in \K_p$, let us say that $A \leq_p B$ when the
  only connected components of $B$ with vertices in $A$ are entirely
  contained in $A$.  It is clear that $(\K_p, \leq_p)$ has free
  amalgamation and $\a$-closures.  In the generic, the closure of a
  finite set will be a minimal set of connected components containing
  that set; the generic thus consists of finite paths separated by
  dense linear orders of vertices.

  $(\K_p, \leq_p)$ does not have $\e$-resolutions: for $A \in \K_p$ and
  $x \in A$, if there are edges in $A$ from $x$ to $x+1$ and $x-1$,
  then $(x, x+1)$ and $(x, x-1)$ will both be biminimal pairs that
  must be in the closure of $\{\, x \,\}$. In fact all biminimal pairs
  will be of the form $(x, x y)$ where $y =  \pm 1$ and
  there is an edge from $x$ to $y$.  Thus we have three
  possibilities for an $\e$-companion of $(\K_p, \leq_p)$: we can choose
  $(x-1) \epo_x (x+1)$, or $(x+1) \epo_x (x-1)$, or the two
  possibilities could be equivalent.  In the first case, we will have
  $A \leq_\e B$ when for every $x \in A$:
  \begin{itemize}
  \item If $(x + 1) \in B$ and $B \models E(x, x+1)$, then either $(x+1) \in
    A$ or $(x-1) \in A$.
  \item If $(x - 1) \in B$ and $B \models E(x, x-1)$, then $(x-1) \in
    A$.
  \end{itemize}
  Thus if $A \subseteq B$ with $|A| > 1$ and both are connected, a
  resolution of $A$ in $B$ can be formed by adding all the vertices in
  $B$ which are less than those in $A$ (in the case that
  $A$ is a singleton $x$ and $B = \{\, x, x+1 \,\}$ we would have $B$
  as the only resolution).

  The $(\K_\e, \leq_\e)$-generic will thus consist of countably many
  infinite paths of order type $(\omega, <)$ separated by dense linear
  orders of vertices.  If we had chosen to have $(x+1) \epo_x (x-1)$,
  we would get paths of order type $\omega^*$ and by having them
  equivalent we would get countably many paths of order type
  $(\Z, <)$.

  It is clear that these are all the $\e$-companions for $(\K, \leq)$
  and that $(\K, \leq)$ is the $\a$-companion of each of these (one
  need only notice that the biminimal pairs are pairs $(x, x \pm 1)$
  with an edge between them in each case).

\end{ex}

\begin{ex}
  Let $(\K_d, \leq_d)$ be the class of distanced graphs as above.
  Then biminimal pairs in $(\K_d, \leq_d)$ consist of pairs
  $(\{\, x_0, x_1 \,\}, p)$ where $p$ is a simple path between $x_0$
  and $x_1$ and $x_0, x_1$ are not joined by an edge.  We can form an
  $\a$-companion $(\K_C, \leq_C)$ by saying $A \leq_C B$ if for
  $a_0, a_1 \in A$ every simple path from $a_0$ to $a_1$ of length at
  least 2 in $B$ is also in $A$.  

  We have that $(\K_d, \leq_d)$ is an $\e$-companion of
  $(\K_C, \leq_C)$.  We could attempt to form another $\e$-companion
  $(\K_m, \leq_m)$ by saying $A \leq_m B$ if for every
  $a_0, a_1 \in A$ the {\it longest} simple path from $a_0$ to $a_1$
  in $B$ is no longer than the longest simple path from $a_0$ to $a_1$
  in $A$.  However, the example illustrated below shows that this
  would not be a free amalgamation class.  Let $A$ consist of the five
  vertices on the solid ellipse, let $B = A \cup \{\, b \,\}$ and let
  $C = A \cup \{\, c \,\}$.  Then by inspection we see that
  $A \leq_m B, A \leq_m C$ but $C \not \leq_m B \oplus_A C$ since there is a
  simple path of length $5$ from $c$ to $a$ in the latter structure.

  \smallskip
  \begin{center}
    \includegraphics[width=1in]{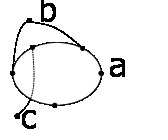}
  \end{center}

\end{ex}

\begin{ex}\label{hamilton-ex}
  Let $\K_H$ be the class of all finite graphs; and for $A \in \K_H$
  define a ``(non)-Hamiltonicity'' function $H$ by
  $H(A) = |A| - \lambda(A)$ where $\lambda(A)$ is the number of
  vertices in the longest simple path in $A$ and for $A \subseteq B$ let
  $H(B/A) = H(B) - H(A)$.  Let us say that $A \leq_H B$ if for every
  $A_0 \subseteq A$ and $B_0 \subseteq B \setminus A$ we have
  $H(A_0 B_0 / A_0) \geq 0$.  Then $(\K_H, \leq_H)$ clearly satisfies A1
  through A6.  
  
  It is not difficult to see that a biminimal pair in this class will
  be a pair $(\{\, x, y \,\}, p)$ where there is no edge from $x$ to
  $y$ and $p$ is a simple path between $x$ and $y$.  Since this is the
  same notion of a biminimal pair that is used in the last example and
  since $(\K_H, \leq_H)$ has $\a$-closures, we see that this class is
  isomorphic to $(\K_C, \leq_C)$.

  While $(\K_C, \leq_C)$ and $(\K_H, \leq_H)$ are isomorphic classes,
  their different definitions allows for some insights into the
  generic $G_{CH}$ of these classes. For example, it is easy to see from
  the definition of $(\K_H, \leq_H)$ that every finite graph $A$ without a
  Hamiltonian path embeds into $G_{HC}$ in such a way that no finite
  extension of $A$ contains a Hamiltonian path either.  The definition
  of $(\K_C, \leq_C)$ makes it clear that any path between vertices
  $\{\, x,y \,\}$ which are joined by an edge must be in the closure
  of  $\{\, x,y \,\}$.  
  
  The definition of $(\K_H, \leq_H)$ also suggests another
  $\e$-companion for these classes besides $(\K_d, \leq_d)$:  we simply treat all paths
  from $x$ to $y$ as equal in the induced partial order.    

\end{ex}

\section{Properties of Classes with $\e$-Resolutions}
In a class with $\exists$-closures, for any $M$ in which $\K$ is
cofinal and $A \finsubs M$, there are two notions of closure in play:
(non-unique) resolutions of $A$ and the unique maximal closure
$\mcl_M(A)$.  We will examine the model-theoretic information given by each
of these closures.

It will be useful in this analysis to make use of some definitions and
results from \cite{brody-intrinsic-trans}.  First we define intrinsic
formulae which will give an approximate description of the maximal
closure of a finite set.

\begin{defn}
\def\a{\bar a} For a fixed amalgamation class $(\K, \leq)$ a
  $0$-intrinsic formula over a finite tuple $\a$ is a formula of the
  form
\begin{equation}
  \label{kint-0}
  \phi(\x; \a) := \Delta_B(\a \x)
\end{equation}
where $(\a_0, B)$ is a minimal pair for $\a_0 \subseteq \a$ and
$\Delta_B(\a \x)$ asserts that $\a_0 \x$ is isomorphic to $B$.

Having defined $k$-intrinic formulae, we define a $k+1$-intrinsic
formula to be of the form
\begin{equation}
  \label{kint-ind}
  \def\w{\bar w}
  \phi(\x; \a) := \Delta_B(\a \x) \land \bigland_{i < m} \exists \w_i
  \phi_i(\w_i; \a \x) \land \bigland_{j < n} \neg \exists \z_j
  \psi_j(\z_j; \a \x)
\end{equation} where again $(\a_0, B)$ is a minimal pair for some $a_0
\subseteq A$ and the
$\phi_i, \psi_j$ are $k$-minimal formulae. 

We will call a formula $\phi(\x; \a)$ {\it intrinsic} over $\a$ when
it is $k$-intrinsic over $\a$ for some $k \in \omega$
\end{defn}

We get a full elementary description of a finite set's maximal closure
by passing to the {\it closure type}.

\begin{defn}
  \def\w{\bar w} \def\m{\bar m} \def\a{\bar a} Let $(\K, \leq)$ be an
  amalgamation class and let $\C$ be a monster model for the theory of
  the $(\K, \leq)$-generic (that is, a ``sufficiently'' saturated
  model of the theory of the ${(\K, \leq)}$ generic).  For any fixed
  tuple $\a \finsubs \C$, the closure-type of $\a$, denoted
  $\cltp(\a)$, is defined by \[
  \begin{aligned}
    \cltp(\a) = &\{\, \exists \x \phi(\x; \w) : \phi
    \text{ is intrinsic over } \w, \C \models \exists \x \phi(\x; \a) \,\}
    \cup \\ &\{\, \neg \exists \y \psi(\y; \w) : \psi
    \text{ is intrinsic over } \w, \C \models \neg \exists \y \psi(\y;
    \a) \,\} 
  \end{aligned}
  \]
  
  For fixed $\a$ and $M \subseteq \C$ any set, the closure type of $\a$ over $M$
  is defined by $\cltp(\a / M) = \bigcup_{\m \finsubs M} \cltp(\a \m)$.
\end{defn}

\begin{lem}\label{cltp-fin}
  \def\a{\bar a} Suppose $(\K, \leq)$ is a full amalgamation class and
  $\C$ is an associated monster model.  Suppose $\a, \b \finsubs \C$
  are finite tuples with $\cltp(\a) = \cltp(\b)$.  Fix
  $a_0 \in \C \setminus \cl(\a)$, and let $\Sigma$ be a finite
  fragment of $\cltp(\a a_0)$.  Then there is a {\bf finite}
  $D \subseteq \cl(\a a_0)$ such that if $f: D \rightarrow \C$ is an
  embedding of $D$ with $f: \a \mapsto \b$ and
  $\cl_\C(f(D)) = \cl_\C(\b) \oplus_{\b} f(D)$, then there is some
  $b_0 \in f(D)$ so that $\C \models \Sigma(\b b_0)$.
\end{lem}

\begin{prop}\label{closed-elem}   Let $\C$ denote a sufficiently saturated model of the theory of the
  $(\K, \leq)$-generic, where $(\K, \leq)$ is a full amalgamation
  class.  Let $M \subseteq \C$ with $|M| < |\C|$.  Let
  $a, b \in \C$ and suppose $\cltp(a/M) = \cltp(b/M)$.  Then
  $\tp_\C(a/M) = \tp_\C(b / M)$
\end{prop}

\begin{proof}
  These are proved in \cite{brody-intrinsic-trans} under the
  assumptions of  A1--A6.  It is an easy check that A6 is not used in
  the proofs.
\end{proof}

One natural question one might ask is what a resolution of $A$ tells
us about how $A$ sits inside of $M$ versus what $\mcl_M(A)$ tells us.
A resolution gives the possibilities for $\mcl_M(A)$.  Somewhat
surprisingly, it turns out that {\it in the generic}, they provide the
same information in that a single resolution of $A$ will determine
$\mcl_M(A)$.  Under full amalgamation and $\exists$-closures we make
this precise in the following.

\begin{prop}
  Let $(\K, \leq)$ be any full amalgamation class, let $\C$ be a
  sufficiently saturated model of the theory of the
  $(\K, \leq)$-generic, and let $M \subseteq \C$ be the $(\K, \leq)$-generic.
  \begin{enumerate}
  \item Let $A, A'$ be finite substructures of $\C$ with an isomorphism
    $f:  A \to A'$ which extends to an isomorphism 
    ${\hat f}:  \mcl_\C(AM) \isom \mcl_\C(A'M)$.  Then $\tp_\C(A/M) =
    \tp_\C(A' / M)$.  

    Further, it suffices that the structure $(\mcl_\C(AM), A)$ (this is
    the structure with universe $\mcl_\C(AM)$ and parameter set $A$) be elementarily
    equivalent to $(\mcl_\C(A'M), A')$.
  \item Let $A, A' \finsubs M$.  If there are finite resolutions
    $B, B'$ of $A,A'$ respectively and an isomorphism $f: B \isom B'$
    with $f(A) = A'$, then $\tp_M(A) = \tp_M(A')$
  \end{enumerate}
\end{prop}

\begin{proof}
  The first follows from Proposition \ref{closed-elem}.  For the
  second, note that $B \leq \mcl_M(A)$ and $B' \leq \mcl_M(A')$.  Thus
  by injectivity and finite closures in $M$ we can easily construct a
  winning strategy for the player $\e$ in any finite length
  Ehrenfeucht-Fraisse game.  By the Fraisse-Hintikka theorem, we have
  $\tp_M(A) = \tp_M(A')$
\end{proof}

It is also fairly easy to see that in classes with $\e$-resolutions,
finite sets have countable resolutions.  As a result, the generic of
such a class will have a superstable theory.

\begin{thm}
  Let $(\K, \leq)$ satisfy A1--A5 and have $\e$-resolutions.  Then
  \begin{enumerate}
  \item If $M$ is any structure with $\K$ cofinal in $M$ and $A
    \finsubs M$, there is a countable resolution $B$ of $A$ in $M$.
  \item If $(\K, \leq)$ also has  full amalgamation, then the theory of the $(\K, \leq)$-generic is superstable.
  \end{enumerate}
\end{thm}

\begin{proof}
  For the first statement, for each $X \subseteq A$, choose
  $Y_X \finsubs M$ so that $(X, Y_X)$ is a biminimal pair and $Y_X$ is
  minimal in the $\epo_X$ ordering.  Let
  $A_1 := \bigcup_{X \subseteq A} Y_X$ and note that $A_1$ is finite.
  Thus we can iterate this procedure and define
  $A_{n+1} := \bigcup_{X \subseteq A_n} Y_X$ where $Y_X \finsubs M$ is
  such that $(X, Y_X)$ is an $\epo_X$-minimal biminimal pair.  It is
  clear that letting $B = \bigcup_{n \in \omega} A_n$ we have
  $B \leq M$.

  \def\m{\bar m}
  For the second statement, let $M$ be a model of the theory of the
  generic with $|M| \geq 2^{\aleph_0}$.  Let $p \in S_1(M)$, let $\C$ be
  a $|M|^+$-saturated elementary extension of $M$, and let $a \in \C$
  be a realization of $p$.  Note that for $\m \finsubs M$ there are at
  most $2^{\aleph_0}$ possibilities for $\cltp(a \m)$, hence there are
  at most $2^{\aleph_0} \cdot |M| = |M|$ possibilities for $\cltp(a /
  M)$.  Since the latter determines $\tp(a/M)$, we have $|S_1(M)| \leq
  |M|$. 
\end{proof}

From \cite{moss} we have that the theory of the generic of the class
$(\K_d, \leq_d)$ of distanced graphs is not $\omega$-stable; it is
thus strictly superstable\footnote{It is not small, which is of
  interest since Baldwin has conjectured that no strictly superstable
  generic structure exists which is also small (assuming A1--A6); see
  \cite{baldwin-easter} }.

We have the following lemma about full amalgamation
\begin{lem}\label{infinite-full-amalg}
  Let $(\K, \leq)$ be a full amalgamation class with $\e$ closures.
  Then for any $M$ which is ($\K, \leq)$-injective, if $A \subseteq C$
  and $A \leq B$ (with $B \in \K$) then $C \leq B \oplus_A C$.
\end{lem}

\begin{proof}
  Let $(X, Y)$ be a biminimal pair with $X \subseteq C, Y \subseteq B
  \oplus_A C$.  Choose a finite $C_0$ so that $Y \subseteq B \oplus_A
  C_0$.  Then $C_0 \leq C_0 \oplus_A B$ by full amalgamation, so we
  must have $Y' \epo_X Y$ with $Y' \subseteq C_0 \subseteq C$.
\end{proof}

We also show that under fairly modest assumptions, the class of $(\K, \leq)$-injective
structures is elementary.

\begin{defn}
  A class $(\K, \leq)$ with $\e$-resolutions has finitary $\e$-resolutions if for
  every $X \in \K$, the partial order induced by $\epo_X$ is a
  well-ordering and every equivalence class is finite.
\end{defn}

\begin{lem}
  If $(\K, \leq)$ has finitary $\e$-resolutions, then for any
  $B \in \K$ there is a universal formula $k_B(\x)$ such that for $M$
  in which $\K$ is cofinal, $M \models k_B(\b)$ exactly when $\b$ is
  isomorphic to $B$ (under some fixed ordering of the ordering of the
  elements of $B$) and is closed in $M$.
\end{lem}

\begin{proof}
  Let $k_B$ assert that $\x$ has the quantifier free type of $B$ and
  also, for every $\x_0 \finsubs \x$, that there does not exist $\y$
  with $(\x_0, \y)$ a biminimal pair and $\y$ smaller under
  $\epo_{\x_0}$ then any $\epo_{\x_0}$-minimal biminimal extension
  which occurs in $\x$.  The set off all such extensions is finite by
  finitariness of the $\e$-resolutions.
\end{proof}

We note that the class of distanced graphs has finitary
$\e$-resolutions as does the class of Example \ref{hamilton-ex}.  The
class from Example 4.14, however, does not.

\begin{lem}
  Let $(\K, \leq)$ have finitary $\e$-resolutions.  Let $\Sigma_I$ be the
  following sentences:
  \begin{itemize}
  \item $\forall \x \neg
    \Delta_B(\x)$ for finite $B \not \in \K$ and $\Delta_B(\x)$ the quantifier-free type of $B$.
  \item $\forall \x k_A(\x) \rightarrow \exists \y k_B(\x \y)$
    for $A \leq B$ and $k_A, k_B$ as guaranteed by the previous lemma.
  \end{itemize}
  The $\Sigma_I$ axiomatizes the class of $(\K, \leq)$-injective
  structures.  
\end{lem}

\begin{proof}
  Clear.
\end{proof}

\subsection{Transfer}
We note that for a given class with $\forall$-closures, any
corresponding $\exists$-companion will have a generic that is no more
complex then that of the original class.  This allows us to prove
various transfer theorems.

Throughout, fix a full amalgamation class $(\K, \leq_\forall)$
with $\forall$-closures and let $(\K, \leq_\exists)$ be an
$\exists$-companion.  Let $G_\forall, G_\exists$ denote the
corresponding generics.

Note that for $A, B \in \K$, if $A \leq_\a B$ then {\it a fortiori} $A
\leq_\e B$.  This has the following consequences
\begin{lem}
  Let $G_\forall, G_\exists$ respectively denote the $(\K, \leq_\a),
  (\K, \leq_\e)$ generics.  Then 
  \begin{enumerate}
  \item   $G_\a$ embeds as a $\leq_\e$-strong substructure of $G_\e$.
  \item If $M_\e$ is any $(\K, \leq_\e)$ injective structure, then
    $G_\e$ embeds as a $\leq_\e$-strong substructure of $M_\e$
  \item If $M_\e$ is any $(\K, \leq_\e)$ injective structure, then
    $G_\a$ embeds as a $\leq_\e$-strong substructure of $M_\e$
  \end{enumerate}
\end{lem}

\begin{proof}
  For the first item, write $G_\a = \bigcup_{i < \omega} A_i$ where
  $A_i \leq_\a A_{i+1}$.  Then by induction $A_i$ embeds as a $\leq_\e$-strong
  substructure of $G_\e$.  Thus  $G_\a$ embeds into $G_\e$, say with
  image $G_\a'$; if $(X,Y)$
  is a biminimal pair in $G_\e$ with $X \finsubs G_\a'$, then for
  some $i$ $X \subseteq A_i'$ (the image of $A_i$).  Since $A_i'
  \leq_\e G_\e$, we must have some $Y' \epo_X Y$ with $Y' \subseteq Y$.

  The same argument modified to use a decomposition $G_\e = \bigcup_{i
    < \omega} B_i$ ($B_i \leq_\e B_{i+1}$) establishes the second item,
  while the third is an immediate consequence of the first two.  
\end{proof}

\begin{thm} 
  Assume that $(\K, \leq_\a)$ and $(\K, \leq_\e)$ both have full
  amalgamation.  Let $\Th(G_\a)$ denote the theory of the
  $(\K, \leq_\a)$-generic and similarly for $\Th(G_\e)$.  Let
  $M_\e \models \Th(G_\e)$.  Then there is some
  $M_\a \models \Th(G_\a)$ with $|M_\a| = |M_\e|$ so that
  $|S_1(M_\e)| \leq |S_1(M_\a)|$, where $S_1(X)$ denotes the set of
  complete 1-types over parameter set $X$.  In particular, if
  $\Th(G_\forall)$ is stable ($\omega$-stable, superstable, etc), then
  so is $\Th(G_\exists)$
\end{thm}

\begin{proof}
  \def\m{\bar m} Let $C \cup D$ be a set of $|M_\e| + \aleph_0$ new constants, and let
  $S$ be a set of sentences asserting that $C$ is isomorphic to $M_\e$
  along with sentences asserting that for $A \in \K$, some elements of
  $D$ form an $\a$-closed copy of $A$ in any model of $S$.  Since
  $(\K, \leq_\a)$ has the same age as $(\K, \leq_\e)$, by compactness
  there is some $M_\a \models \Th(G_\a) \cup S$.  Let
  $h: M_\e \to M_\a$ be an embedding.  We define an injection
  $f: S_1(M_\exists) \to S_1(M_\forall)$. Fix $p \in S_1(M_\exists)$
  and let $a_p$ be a realization in some $|M_\e|$-saturated extension
  $N$ of $M_\e$.  Let $C_p = \cltp(a_p / M_\exists)$, and for
  $\phi(x; \m_\e) \in C_p$ let $h(\phi) = \phi(x; h(\m_\e))$ and let
  $P(x) = \bigcup_{\m_\e} \{\, h(\phi) \,\}$
  By Lemma \ref{cltp-fin}, each finite fragment $\Sigma$ of $P(x)$ is
  associated with a structure $D$ such that any $\leq_\e$-strong
  embedding of $D$ into $M_\a$ will contain a realization of $\Sigma$.
  Since the age of $M_\a$ is the same as that of $M_\e$ and since
  $M_\a$ embeds every element of $\K$ as closed substructure, there
  will be such an embedding.  Thus by compactness, $P(x)$ extends to a
  completion in $S_1(M_\a)$.  Letting $f(p)$ be any such completion,
  we note that since different types in $S_1(M_\e)$ differ in their
  closure-types, and since this difference must be witnessed by a
  finite subset of $M_\e$, we have that $f$ is an injection as
  required.
\end{proof}

% \begin{thm}
%   We have $|S_1(G_\exists)| \leq |S_1(G_\forall)|$, where $S_1(X)$
%   denotes the set of complete 1-types over parameter set $X$.  In
%   particular, if $\Th(G_\forall)$ is stable ($\omega$-stable,
%   superstable, etc), then so is $\Th(G_\exists)$
% \end{thm}

% \begin{proof}
%   \def\g{\bar g} We define an injection
%   $f: S_1(G_\exists) \to S_1(G_\forall)$. Fix $p \in S_1(G_\exists)$
%   and let $a_p$ be a realization in some countably saturated model
%   $M$ of $\Th(G_\exists)$.  Let $C_p = \cltp(a_p / G_\exists)$,
%   enumerate $G_\e$ as $\{\, g_i : i \in \omega \,\}$ and let
%   $\g_n := \{\, g_i : i < n \,\}$.  Then
%   $C_p = \bigcup_{n \in \omega} \cltp(a_p \g_n)$.  Let
%   $P(x) := \bigcup_n \{\, \exists y_0 \ldots y_{n-1} \phi(x y_0 \ldots
%   y_{n-1}) \land \Delta_{\g_n}(y_0 \ldots y_{n-1}) : \phi \in
%   \cltp(a_p \g_n) \,\}$.
%   By Lemma \ref{cltp-fin}, each finite fragment $\Sigma$ of $P(x)$ is
%   associated with a structure $D$ such that any $\leq_\e$-strong
%   embedding of $D$ into $G_\a$ will contain a realization of $\Sigma$.
%   Since the age of $G_\a$ is the same as that of $G_\e$, the
%   injectivity of $G_\a$ guarantees such an embedding.  Thus by
%   compactness, $P(x)$ extends to a completion in $S_1(G_\a)$.  Letting
%   $f(p)$ be any such completion, we note that since different types in
%   $S_1(G_\e)$ differ in their closure-types, and since this difference
%   must be witnessed by a finite subset of $G_\e$, we have that $f$ is
%   an injection as required.
% \end{proof}

\section{Moss Structures}
In this section we discuss the existence of Moss structures for
amalgamation classes.  We will see that the question of the
existence of such structures is quite straightforward in classes with
$\a$-closures but more delicate for classes with $\e$-resolutions.

\begin{defn}\label{defn-moss-struct}
  Let $(\K, \leq)$ be any amalgamation class.  A $(\K, \leq)$ {\it
    Moss structure} is a structure $M$ which is $(\K, \leq)$-injective but
  which does not have finite closures.  That is, there is some $A
  \finsubs M$ such that there is no $B$ with $A \subseteq B
  \finsubs M$ and $B \leq M$
\end{defn}

We first note that the existence of Moss structures is easily
established in classes $(\K, \leq)$ with $\a$-closures, unless every
structure in which $\K$ is cofinal has finite closures (in which case,
following \cite{bal-shi}, we say that $(\K, \leq)$ has finite closures).

\begin{prop}\label{yes-moss}
  If $(\K, \leq)$ is a full amalgamation class with $\forall$-closures
  and without finite closures, then there is a Moss structure which
  satisfies the theory of the $(\K, \leq)$-generic.
\end{prop}

\begin{proof}
  By our supposition, there is an infinite chain of minimal pairs
  $(X_i, X_{i+1})$ with $X_i \neq X_{i+1}$.  Let $p$ be the type which
  asserts that a copy of $\bigcup_{i < \omega} X_i$ exists in any
  realization of $p$.  Then $p$ is finitely satisfied in the generic,
  so that some elementary extension of the generic realizes $p$.
\end{proof}

We will not give a general account of the existence of such structures
in classes with $\e$-resolutions.  We do note that under full
amalgamation such structures are obstructed by a countably infinite
version of injectivity.

\begin{thm}\label{no-moss}
  Suppose $(\K, \leq)$ has $\e$-resolutions.  Let
  us say that $M$ is $(\K, \leq)$ {\it countably injective} when for
  any countable structures $C, D$, if $C \leq D$ and $f: C \into M$ is a
  strong embedding, then $f$ extends to a strong embedding
  ${\hat f}: D \into M$.  Then if
  \begin{enumerate}
  \item $M$ is countably injective; OR
  \item\label{finitary} $M$ is a model of the theory of the $(\K, \leq)$-generic, and
    $(\K, \leq)$ has finitary $\e$-resolutions; OR
  \item $M$ is an $\omega$-saturated model of the $(\K, \leq)$-generic
  \end{enumerate}
  then   $M$ is not a Moss structure.
\end{thm}

\begin{proof}
  For all cases let $A \finsubs M$ and for $X \subseteq A$, choose
  $Y_X \finsubs M$ so that $(X, Y_X)$ is a biminimal pair and $Y_X$ is
  minimal in the $\epo_X$ ordering.  Let $B := A \cup \bigcup_{X \subseteq A} Y_X$
  \begin{enumerate}
  \item Let $C$ be a countable resolution of $A$ note that $C \leq
    CB$ since $C$ is closed.  Thus
     by countable injectivity $B$ embeds strongly into
    $M$ over $C$.  Letting $B'$ be the image of $B$ under such an
    embedding, we have that $B'$ is a finite resolution of $A$ in $M$.
  \item Let $h:  B \into G$ be any embedding of $B$ into the $(\K,
    \leq)$-generic and let $C_h$ be any finite resolution of $h(A)$.
    Note that $C_h \leq C_h h(B)$ so
    that $h$  extends to a strong embedding $B \into G$.  Thus every
    embedding of $A$ into the generic which extends to a copy of $B$
    also extends to a closed copy of $B$.  Since the class is
    finitary, this is a first-order statement and thus holds in $M$ as
    well.  
  \item The argument is similar to the previous one.  Now, for
    $n \in \omega$ let $\sigma_n$ be a sentence which says of its
    models $N$ that for any $h: B \into N$, $h$ can be extended to
    some ${\hat h}: B \into N$ such that ${\hat h}(B)$ is closed in
    every cardinality $|B| + n$ extension of ${\hat h}(B)$.  As
    before, we must have $M \models \sigma_n$ for all $n$ (every
    $h: B \into G$ gives rise to a strong ${\bar h}: B \into G$ over
    $h(A)$ so that $G \models \sigma_n$, hence $M$ does as well.)  Thus
    there is a type $p(\x)$ which asserts that $A \x$ is isomorphic to
    $B$ and is closed.  By saturation, $p$ is realized in $M$.
  \end{enumerate}
\end{proof}

Let $(\K_\e, \leq_\e)$ be the $\e$-companion to $(\K_p, \leq_p)$ from
Example \ref{path-example} in which $(x-1) \epo_x (x+1)$ and let $M$
consist of countably many infinite paths with order types $(\Z, <)$
separated by dense linear orders of vertices; then $M$ will be a Moss
structure.  The only finite closed sets are sets of vertices of degree
$0$ so we have injectivity.  On the other hand, no finite subset of
one of the paths of order-type $(\Z, <)$ will have a finite closure.
In fact, noting that an $\omega$-saturated model of the $(\K_\e,
\leq_\e)$-generic will have a path with order type $(\Z, <)$, this
example essentially shows that a Moss structure exists which is
elementarily equivalent to the generic.  This illustrates the
necessity of requiring finitary closures in the above theorem, since
it is a quick check that $(\K_\e, \leq_\e)$ has full amalgamation.

We conclude with a restatement and proof of the main theorem

\mainthm*

\begin{proof}
  For the first item, we need only note that $\omega$-saturation of
  the $(\K, \leq)$-generic corresponds to the class not having finite
  closures (this is proved in \cite{bal-shi}) and cite Proposition
  \ref{yes-moss}.  The second item is precisely Theorem
  \ref{no-moss}.\ref{finitary}.

\end{proof}

\bibliographystyle{plain}
\bibliography{biblio}

\end{document}

%% file: latex-header-arxiv.tex
\def\C{\mathfrak C}

\def\K{\bold K}
\def\Z{\bold Z}

\def\emptyset{\varnothing}

\def\L{\mathscr{L}}

\def\into{\hookrightarrow}
\def\sinto{\overset{\mathsmaller{\leq}}{\into}}
\def\bigland{\bold{\bigwedge}}

\def\cl{\mathop{\roman{cl}}}

\def\fraisse{Fra\"{\i}ss\'e}

\newcommand{\finsubs}{\mathrel{\subseteq_\omega}}
\newcommand{\epo}{\mathrel{\preceq}}
\def\x{\bar x}
\def\y{\bar y}
\def\z{\bar z}

\def\isom{\cong}

\def\b{{\bar b}}

\let\cl\undefined

\DeclareMathOperator{\cl}{cl}
\DeclareMathOperator{\cltp}{cltp}
\DeclareMathOperator{\mcl}{mcl}

\DeclareMathOperator{\tp}{tp}

\DeclareMathOperator{\Th}{Th}

\newtheorem{thm}{Theorem}[section]

\newtheorem{cor}[thm]{Corollary}
\newtheorem{lem}[thm]{Lemma}

\newtheorem{notn}[thm]{Notation}

\newtheorem{prop}[thm]{Proposition}

\theoremstyle{remark}

\newtheorem{ex}[thm]{Example}

\theoremstyle{definition}
\newtheorem{defn}[thm]{Definition}

%% file: eres-arxiv.bbl
\begin{thebibliography}{1}

\bibitem{baldwin-easter}
John~T Baldwin.
\newblock Problems on pathological structures.
\newblock In {\em Proceedings of 10th Easter conference in model theory,
  Wendisches Rietz}, pages 1--9. Citeseer, 1993.

\bibitem{bal-shel}
John~T. Baldwin and Saharon Shelah.
\newblock Randomness and semigenericity.
\newblock {\em Trans. Amer. Math. Soc.}, 349(4):1359--1376, 1997.

\bibitem{bal-shi}
John~T. Baldwin and Niandong Shi.
\newblock Stable generic structures.
\newblock {\em Ann. Pure Appl. Logic}, 79(1):1--35, 1996.

\bibitem{brody-intrinsic-trans}
Justin Brody.
\newblock Full amalgamation classes with instrinic transcendentals.
\newblock arXiv:math.lo/1512.03888.

\bibitem{mcl-kuek}
D.~W. Kueker and M.~C. Laskowski.
\newblock On generic structures.
\newblock {\em Notre Dame J. Formal Logic}, 33(2):175--183, 1992.

\bibitem{kuek}
David Kueker.
\newblock Homogeneous-universal graphs with respect to isometric maps.

\bibitem{mcl}
Michael~C. Laskowski.
\newblock A simpler axiomatization of the {S}helah-{S}pencer almost sure
  theories.
\newblock {\em Israel J. Math.}, 161:157--186, 2007.

\bibitem{moss}
Lawrence~S Moss.
\newblock Distanced graphs.
\newblock {\em Discrete mathematics}, 102(3):287--305, 1992.

\bibitem{wagner}
Frank~O. Wagner.
\newblock Relational structures and dimensions.
\newblock In {\em Automorphisms of first-order structures}, Oxford Sci. Publ.,
  pages 153--180. Oxford Univ. Press, New York, 1994.

\end{thebibliography}
